\documentclass{article} %preprint, review  3p,twocolumn,review

\usepackage{hyperref}
\usepackage[tbtags]{amsmath}
\usepackage{amsfonts}
\usepackage{amssymb}
\usepackage[thmmarks,amsmath]{ntheorem}

\usepackage{tikz}       % draw
\usetikzlibrary{calc,arrows}

{   % define difinition environment
	\theoremstyle{nonumberplain}
	\theoremheaderfont{\bfseries}
	\theorembodyfont{\normalfont}
	\theoremsymbol{$\square$}
	\newtheorem{proof}{Proof.}
}

\allowdisplaybreaks[3]

\newtheorem{definition}{Definition}[section]
\newtheorem{theorem}{Theorem}[section]
\newtheorem{corollary}[theorem]{Corollary}
\newtheorem{lemma}[theorem]{Lemma}
\newtheorem{proposition}[theorem]{Proposition}

\usepackage[bottom=25mm,left=25mm,right=25mm,top=25mm]{geometry}

\begin{document}

%\begin{frontmatter}
	
\title{Some enumerative properties of a class of Fibonacci-like cubes\thanks{This work was supported by NSFC (Grant No. 11761064).}}

\author{ Xuxu Zhao , Xu Wang and Haiyuan Yao\footnote{Corresponding author.}%
	\\ {\small College of Mathematics and Statistics, Northwest Normal University, Lanzhou, Gansu 730070, PR China} \\
\small e-mail addresses: zxxgxxsd@163.com; xuwangac@hotmail.com; hyyao@nwnu.edu.cn}

\date{}

\maketitle

\begin{abstract}

  A filter lattice is a distributive lattice formed by all filters of a poset in the anti-inclusion order. We study the combinatorial properties of the Hasse diagrams of filter lattices of certain posets, so called Fibonacci-like cubes, in this paper. Several enumerative polynomials, e.g.\ rank generating function, cube polynomials and degree sequence polynomials are obtained. Some of these results relate to Fibonacci sequence and Padovan sequence.


\end{abstract}

\section{Introduction}

%  Fibonacci cubes and Lucas cubes have been researched in several terms, the Fibonacci cubes are defined by Hsu in \cite{aHsu1993}, besides, \cite{aRispoC2008} and \cite{aStanl1975} introduced the Fibonacci hypercube and the Fibonacci cube, and Klav\v{z}ar have a survey \cite{aKlavz2013} on Fibonacci cubes. By suitable orientation, Fibonacci and Lucas cubes can be considered as the directed Hasse diagram or covering graph of two distributive lattices \cite{aGansn1982,aMunarCS2001,aYaoZ2015}, namely Fibonacci and Lucas lattices, respectively. The rank polynomial of Fibonacci and Lucas lattices have been given by Munarini and Salvi \cite{aMunarS2002b}. The study have given the structural and enumerative properties of the Fibonacci cubes, and the independence number of the Fibonacci cubes is determined in \cite{eMunarini2002}, the cube polynomial of Fibonacci cubes and Lucas cubes is studied in \cite{aKlavM2012}, as well as maximal hypercubes of them are suggested in \cite{aMolla2012}, and in \cite{aKlavzMP2011} the number of vertices of a given degree have been determined.

  Fibonacci cubes are introduced by Hsu \cite{aHsu93} as a interconnection topology. Lucas cubes is a class of graphs with close similarity with Fibonacci cubes \cite{aMunarCS2001}. Many enumerative properties \cite{aKlavz2013} such as cube polynomials, maximal cube polynomials and degree sequence polynomials of Fibonacci cubes and Lucas cubes are obtained by Klav\v{z}ar \&\ Mollard \cite{aKlavM12}, Sayg{\i} \&\ E\u{g}ecio\u{g}lu \cite{aSaygE18}, Mollard \cite{aMolla12} and Klav\v{z}ar et al.\  \cite{aKlavzMP11}. Munarini and Zagaglia Salvi \cite{aMunarZ02b} found the undirected Hasse diagrams of filter lattices of fences \cite{bStanl11} are isomorphic to Fibonacci cubes and gave the the rank polynomials of the filter lattices. Using a convex expansion, Wang et al.\ \cite{aWangZY18} considered indegree sequences of Hasse diagrams of finite distributive lattices and gave the relation of indegree sequence polynomials between cube polynomials.

    Let $P= (P,\le)\,(p \neq \emptyset)$ is a partially ordered set(poset for short). The \emph{dual} $P^*$ of $P$ by defining $x\le y$ to hold in $P^*$ if and only if $y\le x$ holds in $P$. Let $x\prec y$ denote $x$ is covered by $y$, if $x< y$ and $x\leq z<y$ implies $z=x$. Let $Q$ be a subset of $P$, then $Q$ has an induced order relation from $P$: given $x,y\in Q$, $x\le y$ in $Q$ if and only if $x\le y$ in $P$. The subset $Q$ of the poset $P$ is called convex if $a,b \in Q$, $c \in P$, and $a \le c \le b$ imply that $c \in Q$. A subset $Y$ is a filter (up-set) of $P$ if $y\in Y$ and $y\leq z$ then $z\in Y$. The set $\mathcal{F}(P)$ of all filters of $P$ forms a distributive lattice reordered by anti-inclusion: $Y'\leq Y$ if and only if $Y'\supseteq Y$, namely the filter lattice $\mathcal{F}(P): =(\mathcal{F}(P), \supseteq)$.

   A filter lattice is a distributive lattice formed by filter of a poset in the anti-inclusion order. We study the combinatorial properties of a class of specitial cubes in this paper. These cubes are induced by the Hasse diagram of filter lattices of certain poset, contain Fibonacci cubes as its sub-cubes and have similar properties with Fibonacci cubes, so called Fibonacci-like cubes. We obtained some structural and enumerative properties of the cubes and the distributive lattices, including rank generating functions, cube polynomials, maximal cube polynomials, degree sequence polynomials and indegree sequence polynomials.

\section{Fibonacci-like cubes}
    The fence or ``zigzag poset" $Z_n$\cite{bStanl11} is an interesting poset, with element $\{x_1,\ldots, x_n\}$ and cover relation $x_{2i-1}< x_{2i}$ and $x_{2i}> x_{2i+1}$, and the underlying graph of the Hasse diagram of $\mathcal{F}(Z_n)$ is a Fibonacci cube $\Gamma_n$.  Now we have a new poset modified from fence.
\begin{definition}
   The $S$-fence, denoted by $\phi_{n}$, is a fence-like poset, with element $\{x_1,\ldots, x_n\}$ and cover relations $x_1>x_2>x_3$, $x_2<x_4$ and $x_4>x_5$, for $i\ge 3$, $x_{2i-1}< x_{2i}$ and $x_{2i}> x_{2i+1}$.
\end{definition}

\begin{figure}[!htp]
\centering
\begin{tikzpicture}[scale=0.8]
		\newcommand\ke{4};
		
		\foreach \i in {1,...,\ke}
		{
			\draw (\i-1,0) -- (\i,1) -- (\i,0);
		}
\draw (0,0) -- (0,1) --(0,-1);
		
		\foreach \i in {0,...,\ke}
		{
			\foreach \j in {0,1}
			{
				\filldraw[fill=white] (\i,\j) circle (1.5pt);
			}
		}
		\filldraw[fill=white] (0,-1) circle (1.5pt);
        \node[left] at (0,1) {$x_1$};
		\node[left] at (0,0) {$x_2$};
		\node[left] at (0,-1) {$x_3$};
		\node[below] at (1,0) {$x_5$};
		\node[above] at (1,1) {$x_4$};
		\node[above] at (2,1) {$x_6$};
		\node[below] at (3,0) {$\cdots$};
		\node[below] at (2,0) {$x_7$};
		\node[above] at (3,1) {$\cdots$};
        %\node[above] at (4,1) {$x_{n-1}$};
		\node[below] at (4,0) {$x_{n}$};
        \node at (1.5,-2) {(a)~$n$ is odd};

\end{tikzpicture}\qquad
\begin{tikzpicture}[scale=0.8]
		\newcommand\ke{3};
		
		\foreach \i in {1,...,\ke}
		{
			\draw (\i-1,0) -- (\i,1) -- (\i,0);
		}
		\draw (0,0) -- (0,1)-- (0,-1);
		\draw (\ke,0) -- (\ke+1,1);
		
		\foreach \i in {0,...,\ke}
		{
			\foreach \j in {0,1}
			{
				\filldraw[fill=white] (\i,\j) circle (1.5pt);
			}
		}
		\filldraw[fill=white] (0,-1) circle (1.5pt);
		\filldraw[fill=white] (\ke+1,1) circle (1.5pt);
        \node[left] at (0,1) {$x_1$};
		\node[left] at (0,0) {$x_2$};
		\node[left] at (0,-1) {$x_3$};
		\node[below] at (1,0) {$x_5$};
		\node[above] at (1,1) {$x_4$};
		\node[above] at (2,1) {$x_6$};
		\node[below] at (3,0) {$\cdots$};
		\node[below] at (2,0) {$x_7$};
		\node[above] at (3,1) {$\cdots$};
        %\node[above] at (4,1) {$x_{n-1}$};
		\node[above] at (4,1) {$x_{n}$};
        \node at (1.5,-2) {(b)~$n$ is even};
\end{tikzpicture}
\caption{S-fence $\phi_n$}
\end{figure}
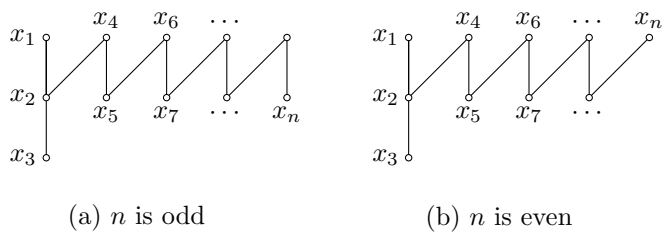

  The Hasse diagram of the filter lattice $\mathcal{F}(\phi_{n})$ can be considered as a directed graph (namely (x,y) is an arc if and only if $y\prec x$). The underlying graph of the Hasse diagram of $\mathcal{F}(\phi_n)$ is a Fibonacci-like cube (FLC for shorter). For convenience, both the Hasse diagram (as a diricted graph) and its underlying graph are denoted by $\Phi_n$ too. Note that $\Phi_0 $ is the trivial graph with only one vertex.

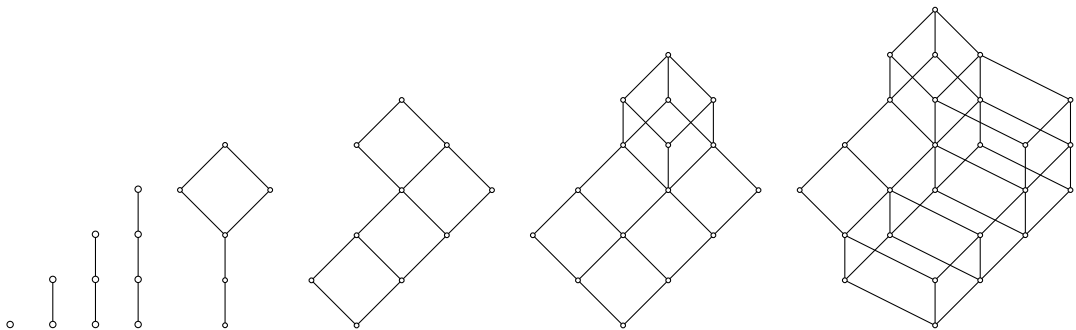
\begin{figure}[!htp]
\centering
\begin{tikzpicture}[scale=0.6]
	\draw (0,0) circle (2pt);
\end{tikzpicture}
\quad
\begin{tikzpicture}[scale=0.6]
	\draw (0,0) -- (0,1);
	\filldraw[fill=white] (0,0) circle (2pt) (0,1) circle (2pt);
\end{tikzpicture}
\quad
\begin{tikzpicture}[scale=0.6]
	\draw (0,0) -- (0,2);
	\filldraw[fill=white] (0,0) circle (2pt) (0,1) circle (2pt) (0,2) circle (2pt);
\end{tikzpicture}
\quad
\begin{tikzpicture}[scale=0.6]
	\draw (0,0) -- (0,3);
	\filldraw[fill=white] (0,0) circle (2pt) (0,1) circle (2pt) (0,2) circle (2pt) (0,3) circle (2pt);
\end{tikzpicture}
\quad
\begin{tikzpicture}[scale=0.6]
	\draw (0,0) -- (1,1) -- (0,2) -- (-1,1) -- cycle -- (0,-2);
	\foreach \i in {0,1}
	{
		\foreach \j in {0,1}
		{
			\filldraw[fill=white] (\i-\j,\i+\j) circle (1.5pt);
		}
		\filldraw[fill=white] (0,\i-2) circle (1.5pt);
	}
\end{tikzpicture}
\quad
\begin{tikzpicture}[scale=0.6]
	\foreach \i in {2,3}
	{
		\draw (\i,\i) -- (\i-2,\i+2);
	}
	\foreach \j in {0,1}
	{
		\draw (\j,\j) -- (\j-1,\j+1) (0-\j,0+\j) -- (3-\j,3+\j);
	}
	\draw (0,4) -- (1,5);
	\foreach \j in {0,1}
	{
		\foreach \i in {0,1,2,3}
		{
			\filldraw[fill=white] (\i-\j,\i+\j) circle (1.5pt);
		}
		\filldraw[fill=white] (\j,\j)++(0,4) circle (1.5pt);
	}
\end{tikzpicture}
\quad
\begin{tikzpicture}[scale=0.6]
	\foreach \i in {0,1,2,3}
	{
		\draw (\i,\i)++(-2,2) -- (\i,\i);
	}
	\foreach \j in {0,1,2}
	{
		\draw (-\j,\j)++(3,3) -- (-\j,\j);
	}
	\foreach \i in {2,3}
	{
		\foreach \j in {1,2}
		{
			\draw (\i-\j,\i+\j+0) -- (\i-\j,\i+\j+1);
		}
		\draw (\i-1,\i+2) -- (\i-2,\i+1+2)
				(\i-2,-\i+7) -- (\i-1,8-\i);
	}
	\foreach \i in {0,1,2,3}
	{
		\foreach \j in {0,1,2}
		{
			\filldraw[fill=white] (\i-\j,\i+\j) circle (1.5pt);
		}
	}
	\foreach \i in {2,3}
	{
		\foreach \j in {1,2}
		{
			\filldraw[fill=white] (\i-\j,\i+\j+1) circle (1.5pt);
		}
	}
\end{tikzpicture}
\quad
\begin{tikzpicture}[scale=0.6]
	\foreach \i in {0,1,2,3}
	{
		\draw (\i,\i)++(-2,2) -- (\i-1,\i+1) -- (\i-1,\i);
	}
	\foreach \j in {1,2}
	{
		\draw (-\j,\j)++(3,3) -- (-\j,\j) (1,\j-2)++(3,3) -- (1,\j-2);
	}
	\draw (-1,0) -- (2,3) (3,3) -- (4,4);
	\foreach \i in {2,3}
	{
		\foreach \j in {1,2}
		{
			\draw (\i-\j,\i+\j+0) -- (\i-\j,\i+\j+1);
		}
		\draw (\i+1,\i) -- (\i+1,\i+1) -- (\i-1,\i+2) -- (\i-2,\i+1+2)
				(\i-2,-\i+7) -- (\i-1,8-\i);
	}
	\foreach \i in {0,1,2,3}
	{
		\foreach \j in {1,2}
		{
			\draw (\i-1,\i+\j-1)++(2,-1) -- (\i-1,\i+\j-1);
		}
		\draw (\i-1,\i)++(2,-1) -- (\i+1,\i);
	}
	\foreach \i in {0,1,2,3}
	{
		\foreach \j in {1,2}
		{
			\filldraw[fill=white] (\i-\j,\i+\j) circle (1.5pt);
		}
		\filldraw[fill=white] (\i-1,\i) circle (1.5pt);
	}
	\foreach \i in {2,3}
	{
		\foreach \j in {1,2}
		{
			\filldraw[fill=white] (\i-\j,\i+\j+1) circle (1.5pt);
		}
	}
	\foreach \i in {0,1,2,3}
	{
		\foreach \j in {-1,0}
		{
			\filldraw[fill=white] (\i+1,\i+\j) circle (1.5pt);
		}
	}
	\filldraw[fill=white] (3,3) circle (1.5pt) (4,4) circle (1.5pt);
\end{tikzpicture}
\caption{The first eight Fibonacci-like cubes $\Phi_0$, $\Phi_1$, \dots, $\Phi_7$}\label{fig:mlc}
\end{figure}

 Let $L$ be a distributive lattice. A convex sublattice (interval) $K$ of $L$ is called a cutting if any maximal chain of $L$ must contain some elements of $K$.
   The convex expansion $L\boxplus K$ of $L$ with respect to $K$ is a distributive lattice on the set $L\cup K'$ ($K^{'}\cong K$ a copy of $K$) with the induced order:
  \begin{eqnarray*}
    x \leq y,  & & \text{ if $x \leq y$ in $L$,}\\
    x^{'}< y, & & \text{ if $x \leq y$ in $L$ with $x \in K$,}\\
    x < y^{'},  & & \text{ if $x < y$ in $L$ with $y \in K$,}\\
    x^{'} \leq y^{'}, & & \text{ if $x \leq y$ in $K$.}\\
  \end{eqnarray*}
As show in Figure~\ref{fig:lplk}, where $\hat1_K$ and $\hat0_K$ denote the maximum element and minimum element of the distributive lattice $K$, respectively (see\cite{aWangZY18}).
  \begin{figure}[!htb]  %%lbpk.tex
	\centering
	\begin{tikzpicture}[scale=1]
	
	\filldraw[fill=lightgray] (0,0) -- (2,1) -- (2,2) -- (0,1) -- cycle;
	\draw (0,1) -- (-0.5,1.866) -- (1.5,2.866) -- (2,2)
	(0,0) -- (1*0.75,-1.732*0.75) -- (1*0.75+2,-1.732*0.75+1) -- (2,1);
	
	\node at (1,1) {$K$};
	\node at (-0.2,0) {$\hat{0}_K$};
	\node at (2.25,2) {$\hat1_K$};
	
	\begin{scope}[xshift=5cm,yshift=0.3535cm]
	\filldraw[fill=lightgray] (0,0) -- (2,1) -- (2,2) -- (0,1) -- cycle
	(0.707,-0.707) -- (2.707,0.293) -- (2.707,1.293) -- (0.707,0.293) -- cycle;
	\draw (0,1) -- (-0.5,1.866) -- (1.5,2.866) -- (2,2)
	(0.707,-0.707) -- (1*0.75+0.707,-1.732*0.75-0.707) -- (1*0.75+2.707,-1.732*0.75+1-0.707) -- (2.707,1-0.707)
	(0,0) -- (0.707,-0.707) (2,1) -- (2.707,0.293) (2,2) -- (2.707,1.293) (0,1) -- (0.707,0.293);
	\node at (1,1) {$K$};
	\node at (1+0.707,1-0.707) {$K'(\cong K)$};
	
	\node at (1.4785,-2-0.3535) {$L\boxplus K$};
	\end{scope}
	\draw[->] (3.25,0.509) -- (4,0.509);
	
	\node at (1.125,-2) {$L$};
	
	\end{tikzpicture}
	\caption{The finite distributive lattices $L$ and $L\boxplus K$}\label{fig:lplk}
\end{figure}
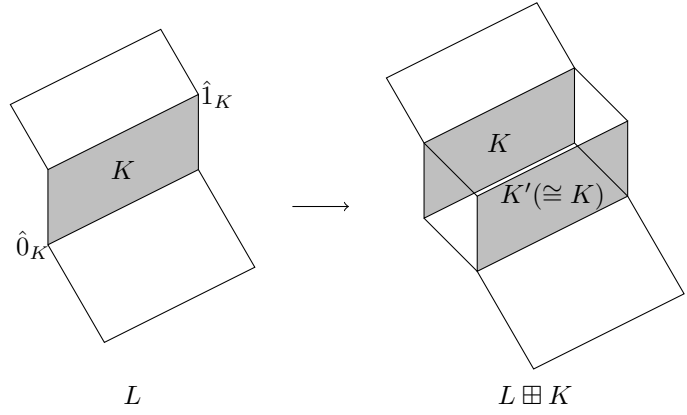

\begin{lemma}\label{lem:conexp}

The filter lattice $\mathcal{F}(P)$ is considered as a convex expansion, that is for every $x\in P$,
	\[
	\mathcal{F}(P)\cong \mathcal{F}(P-x)\boxplus \mathcal{F}(P*x),
	\]
  where $P-x$ and $P*x$ the induced subposets on $P\setminus\{x\}$ and $P\setminus\{y \in P|y \leq x \text{ or } x \leq y\}$, respectively.

\end{lemma}

  It is known that $\Gamma_n \cong \Gamma_{n-1} \boxplus \Gamma_{n-2}  \cong (\Gamma_{n-2} \boxplus \Gamma_{n-2}) \boxplus \Gamma_{n-3} \; (n\ge 4)$ \cite{aKlavz2013}. FLC $\Phi_n$ has similar structure with the Fibonacci cube and matchable Lucas cubes \cite{aWangZY18}, as shown in Figure~\ref{fig:strucf} and Figure~\ref{fig:struc}.

\begin{lemma}\label{lem:struc}
  Let $\Phi_n$ be the $n$-th FLC defined above. Then
  \[
\Phi_n \cong \Gamma_{n-1}^* \boxplus \Gamma_{n-4}^* \cong \left((\Gamma_{n-3}^* \boxplus \Gamma_{n-3}^*) \boxplus \Gamma_{n-4}^*\right) \boxplus \Gamma_{n-4}^* \; (n\ge 5).
\]
or 
  \[
  \Phi_n \cong \Phi_{n-1} \boxplus \Phi_{n-2} \cong (\Phi_{n-2} \boxplus \Phi_{n-2}) \boxplus \Phi_{n-3} \; (n\ge 6).
  \]
\end{lemma}

\begin{proof}
  By definition of $\Phi_n$, the result is easily obtained by putting $x=x_n$ and $x=x_3$ in Lemma \ref{lem:conexp}, respectively.
\end{proof}

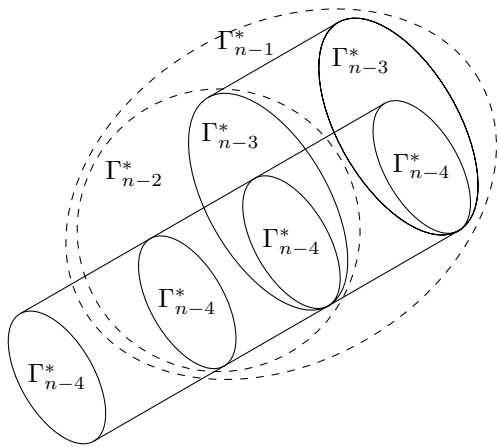
\begin{figure}[!htb]
	\centering
	\begin{tikzpicture}[scale=0.8,rotate=30]
	\draw (-4.5,-2) -- (2.5,-2) (0,2) -- (2.5,2) (-4.5,0.448) -- (2.5,0.448);
	\draw (2.5,0) ellipse (1cm and 2cm)
	(2.5,0) ellipse (1cm and 2cm)
	(0,0) ellipse (1cm and 2cm)
	(-2,-0.776) ellipse (0.612cm and 1.224cm)
	(0,-0.776) ellipse (0.612cm and 1.224cm)
	(2.5,-0.776) ellipse (0.612cm and 1.224cm)
	(-4.5,-0.776) ellipse (0.612cm and 1.224cm);;
	\draw[dashed] (-0.95,0) circle (2.35cm)
	(0.25,0) ellipse (3.8cm and 2.8cm);
	
	\node at (1,2.45) {$\Gamma_{n-1}^*$};
	\node at (-1.7,1.5) {$\Gamma_{n-2}^*$};
	%\node at (5.5,1.25) {$\Gamma_{n-3}$};
	\node at (0,1.25) {$\Gamma_{n-3}^*$};
	\node at (2.5,1.25) {$\Gamma_{n-3}^*$};
	\node at (-2,-0.776) {$\Gamma_{n-4}^*$};
	\node at (0,-0.776) {$\Gamma_{n-4}^*$};
	\node at (2.5,-0.776) {$\Gamma_{n-4}^*$};
	\node at (-4.5,-0.776) {$\Gamma_{n-4}^*$};
	\end{tikzpicture}
	\caption{The structure of $\Phi_n$ given by $\Gamma_n$}\label{fig:strucf}
\end{figure}

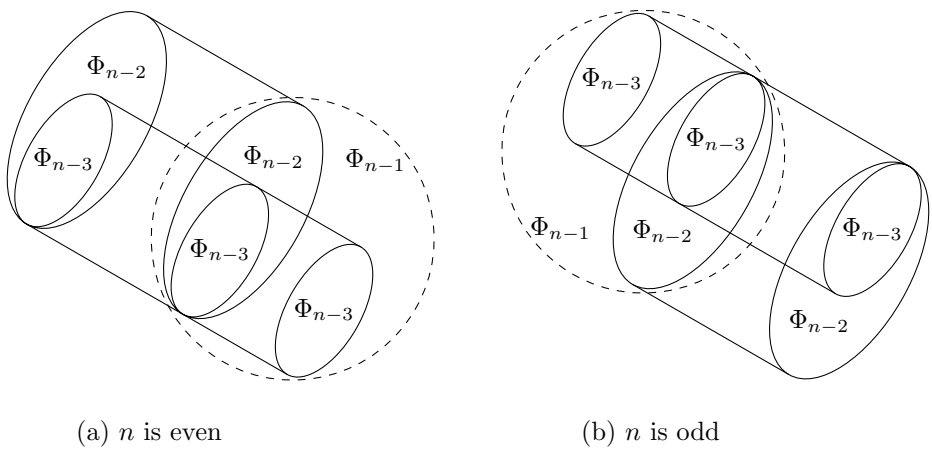
\begin{figure}[!htb]
	\centering
	\begin{tikzpicture}[scale=0.8,rotate=-30]
		\draw (-3,2) -- (0,2) (-3,-2) -- (2,-2) (-3,0.448) -- (2,0.448);
		\draw (-3,0) ellipse (1cm and 2cm)
		(0,0) ellipse (1cm and 2cm)
		(2,-0.776) ellipse (0.612cm and 1.224cm)
		(0,-0.776) ellipse (0.612cm and 1.224cm)
		(-3,-0.776) ellipse (0.612cm and 1.224cm);
		\draw[dashed] (0.95,0) circle (2.35cm);
		
		\node at (1.5,1.8) {$\Phi_{n-1}$};
		\node at (0,1) {$\Phi_{n-2}$};
		\node at (-3,1) {$\Phi_{n-2}$};
		\node at (0,-0.776) {$\Phi_{n-3}$};
		\node at (-3,-0.776) {$\Phi_{n-3}$};
		\node at (2,-0.776) {$\Phi_{n-3}$};
		
		\node at (0.5,-4) {(a)~$n$ is even};
	\end{tikzpicture}\quad
	\begin{tikzpicture}[scale=0.8,rotate=-30]
		\draw (3,-2) -- (0,-2) (3,2) -- (-2,2) (3,-0.448) -- (-2,-0.448);
		\draw (3,0) ellipse (1cm and 2cm)
		(0,0) ellipse (1cm and 2cm)
		(-2,0.776) ellipse (0.612cm and 1.224cm)
		(0,0.776) ellipse (0.612cm and 1.224cm)
		(3,0.776) ellipse (0.612cm and 1.224cm);
		\draw[dashed] (-0.95,0) circle (2.35cm);
		
		\node at (-1.5,-1.8) {$\Phi_{n-1}$};
		\node at (0,-1) {$\Phi_{n-2}$};
		\node at (3,-1) {$\Phi_{n-2}$};
		\node at (0,0.776) {$\Phi_{n-3}$};
		\node at (3,0.776) {$\Phi_{n-3}$};
		\node at (-2,0.776) {$\Phi_{n-3}$};
		
		\node at (1.5,-4) {(b)~$n$ is odd};
	\end{tikzpicture}
	\caption{The recursive structure of $\Phi_n$}\label{fig:struc}
\end{figure}

  By Lemma  \ref{lem:struc} and the fact that $|V(\Gamma_{n})|=F_{n+2}$, we have folloing relation. 
\begin{corollary}
	The number of vertices of $\Phi_n$ is $2F_n$.
\end{corollary}

\section{Enumerative properties}

  Let $F_n$ be the $n$-th Fibonacci number defined by: $F_0=0$, $F_1=1$, $F_n=F_{n-1}+F_{n-2}$, for $n\ge 2$. The generating function of the sequence $\{F_n\}_{n=0}^\infty$ is
  \[
  \sum_{n\ge 0}F_nx^n=\frac{x}{1-x-x^2}.
  \]

   There is an interesting relation between Fibonaci numbers and binomial coefficients:

  \[
  \sum_{k = 0}^{\frac{n}{2}} \binom{n-k}{k} = F_{n+1}.
  \]

  The enumerative properties of Fibonacci cubes and Lucas cubes has been extensively studied $[3\text{-}9]$.
  %\cite{aKlavzMP11,aKlavM12,aMunarCS2001,aMolla12,aMunarZ02b, aKlavz2013,aSaygE18}. 
  In this section, we obtain some enumerative properties of $\Phi_n$, such as  rank generating functions i.e. rank polynomials, cube polynomials, maximal cube polynomial, degree sequences polynomial, indegree and outdegree polynomials. Some results are related to Fibonacci sequences since the number of vertices of $\Phi_n$ equals to $2F_n$. The number of the maximal $k$-dimensional cubes in $\Phi_n$ is a Padovan number.

  Note that hereafter set $\binom{n}{k}=0$ whenever the condition $0 \leq k \leq n$ is invalid for any integers $n$ and $k$.

  The proof of the conclusion about the generating function is similar to the proof of Theorem~\ref{th:gf-AB},and the proof of Propositions \ref{prop:RF-q},~\ref{pro:ou-d} are similar to the proof of Proposition~\ref{prop:gf-H} (see \cite{aKlavM12}, \cite{bStanl11}).

\subsection{Rank generating functions}\label{subsec:rank}

The rank generating function of $\Phi_n$ is $R_n(x) := R(\Phi_n,x)= \sum_{k\ge 0}r_{n,k}x^k$, where $r_{n,k}:=r_k(\Phi_n)$ denoted the number of the elements of rank $k$ in $\Phi_n$.
The first few of $R_n(x)$ are listed.
\begin{align*}
  R_0(x) &= 1 \\
  R_1(x) &= 1+x \\
  R_2(x) &= 1+x+x^2\\
  R_3(x) &= 1+x+x^2+x^3 \\
  R_4(x) &= 1+x+x^2+2x^3+x^4  \\
  R_5(x) &= 1+2x+2x^2+2x^3+2x^4+x^5\\
  R_6(x) &= 1+2x+3x^2+3x^3+3x^4+3x^5+x^6\\
  R_7(x) &= 1+3x+4x^2+5x^3+5x^4+4x^5+3x^6+x^7 \\
  R_8(x) &= 1+3x+5x^2+7x^3+8x^4+7x^5+6x^6+4x^7+x^8\\
  R_9(x) &= 1+4x+7x^2+10x^3+12x^4+12x^5+10x^6+7x^7+4x^8+x^9.
\end{align*}
\begin{lemma}[\cite{aWangZY18}]\label{le:rank}

  	Let $L$ be a finite distributive lattice and $K$ is a cutting of $L$. The rank generating function of $L\boxplus K$ is
	\[
	R(L\boxplus K,x) =
	\begin{cases}
	R(L,x) + x^{h_L(\hat0_K)+1}R(K,x), & \text{if } \hat1_K = \hat1_L; \\
	R(K,x) + xR(L,x), & \text{if } \hat0_K = \hat0_L.
	\end{cases}
	\]
\end{lemma}
where $h_L(x)$ denote the height of $x$ in $L$ for $x \in L$.

By the Lemma~\ref{le:rank} we have
\begin{proposition}\label{prop:rec-R}
	For $n\ge 5$
	\[
	R_n(x) =
	\begin{cases}
	xR_{n-1}(x) + R_{n-2}(x), & \mbox{if } 2\nmid n, \\
	R_{n-1}(x) + x^2R_{n-2}(x), & \mbox{if } 2\mid n.
	\end{cases}
	\]
\end{proposition}
%In other words,
%   \begin{align*}
%   \begin{cases}
%   R_{2m+1}(x)  =xR_{2m}(x)+R_{2m-1}(x), & (m\ge 2),\\
%   R_{2m}(x)=R_{2m-1}(x)+x^2R_{2m-2}(x), & (m\ge 3).
%   \end{cases}
%   \end{align*}
Let $A_m(x) = R_{2m}(x)$ and $B_m(x) = R_{2m+1}(x)$, we have
 \begin{align*}
 \begin{cases}
   A_m(x)=B_{m-1}(x)+x^2A_{m-1}(x), &(m\ge 3),\\
   B_m(x)=xA_{m}(x)+B_{m-1}(x), &(m\ge 2).
 \end{cases}
 \end{align*}
 $A_m(x)$ and $B_m(x)$ have the recurrence relations:
\begin{proposition}\label{prop:rec-AB}
	
	\[
	\begin{cases}
	A_m(x) = (1+x+x^2)A_{m-1}(x) - x^2A_{m-2}(x), & (m\ge 4), \\
	B_m(x) = (1+x+x^2)B_{m-1}(x) - x^2B_{m-2}(x), & (m\ge 2).
	\end{cases}
	\]
\end{proposition}
\begin{proof}
	By Proposition~\ref{prop:rec-R},   for $m\ge 4$,
    \begin{align*}
    A_m(x) &=B_{m-1}(x)+x^2A_{m-1}(x)\\
    &=xA_{m-1}(x)+B_{m-2}(x)+x^2A_{m-1}(x)\\
    &=xA_{m-1}(x)+A_{m-1}(x)-x^2A_{m-2}(x)+x^2A_{m-1}(x)\\
    &=(1+x+x^2)A_{m-1}(x)-x^2A_{m-2}(x).
    \end{align*}
  % For $m\ge 3$,
%	\begin{align*}
%	B_m(x) &= xA_{m}(x)+B_{m-1}(x)\\
%     &=x(B_{m-1}(x)+x^2A_{m-1}(x))+B_{m-1}(x)\\
%     &=xB_{m-1}(x)+x^2(B_{m-1}(x)-B_{m-2}(x))+B_{m-1}(x)\\
%     &=(1+x+x^2)B_{m-1}(x)-x^2B_{m-2}(x).
%	\end{align*}
	For $B_2(x)$, the conclusion is also ture.
Samilarily, the recurrence relation of $B_m(x)$ can be obtained.
\end{proof}
Then, we can derive the generating functions of $A_m(x)$ and $B_m(x)$ by Proposition~\ref{prop:rec-AB}, respectively.
\begin{theorem}\label{th:gf-AB}
	The generating functions of $A_m(x)$ and $B_m(x)$ are
	\[
	\sum_{m\ge 0}A_m(x)z^m = z+ \frac{1-z+z^2}{1-(1+x+x^2)z+x^2z^2}
	\]
	and
	\[
	\sum_{m\ge0}B_m(x)z^m = \frac{1+x-(x+x^2)z}{1-(1+x+x^2)z+x^2z^2},
	\]
	respectively.
\end{theorem}
\begin{proof}
	By the Proposition~\ref{prop:rec-AB},
	\begin{align*}
	\sum_{m\ge 0}A_m(x)z^m &= \sum_{m\ge 4}A_m(x)z^m +A_3(x)z^3 + A_2(x)z^2 + A_1(x)y+A_0(x) \\
	  &= \sum_{m\ge 4}((1+x+x^2)A_{m-1}-x^2A_{m-2})z^m +A_3(x)z^3 + A_2(x)z^2 + A_1(x)z+A_0(x) \\
	  &= (1+x+x^2)z\sum_{m\ge 3} A_m(x)z^m - x^2z^2 \sum_{m\ge 2} A_m(x)z^m A_3(x)z^3+ A_2(x)z^2 + A_1(x)z+A_0(x) \\
	  &= (1+x+x^2)z\sum_{m\ge 0} A_m(x)z^m - x^2z^2 \sum_{m\ge 0} A_m(x)z^m + 1 +x^2z^3-xz^2-x^2z^2.
    \end{align*}
Similarly, the generating function of $B_m$ are obtained.
    %\begin{align*}
	%\sum_{m\ge 0}B_m(x)z^m &= \sum_{m\ge 2}B_m(x)z^m + B_1(x)z+B_0(x) \\
	 %  &= \sum_{m\ge 2}((1+x+x^2)B_{m-1}(x)-x^2B_{m-2}(x))z^m +  B_1(x)z+B_0(x) \\
	  % &= (1+x+x^2)z\sum_{m\ge 1} B_m(x)z^m - x^2z^2 \sum_{m\ge 0} B_m(x)z^m + B_1(x)z+B_0(x) \\
	   %&= (1+x+x^2)z\sum_{m\ge 0} B_m(x)z^m - x^2z^2 \sum_{m\ge 0} B_m(x)z^m + 1+x-xz-x^2z.
  % \end{align*}
 % The proof is completed.
\end{proof}
In addition, $A_m(x)$ and $B_m(x)$ can be obtained by Theorem~\ref{th:gf-AB}.

Let $\binom{n;3}{k}$ denote the coefficient of $x^k$ in $(1+x+x^2)^n$(see \cite{bComte74}), which is
   \[
   \binom{n;3}{k} =\sum_{i=0}^{\lfloor k/2 \rfloor}\binom{n}{k-i}\binom{k-i}{i}.
   \]
   See also the sequence A027907 in the OEIS \cite{Sloan19}. Using Kronecker delta function $\delta$, we have the formula of the coefficient $r_{n,k}$.

\begin{theorem}\label{th:rnk}
     \begin{multline*}
	    r_{2m,k} = \delta_{1,m}\delta_{0,k}+ \sum_{i=0}^{\lfloor \frac m2 \rfloor}(-1)^i\binom{m-i}{i} \binom{m-2i;3}{k-2i} -\sum_{i=0}^{\lfloor \frac{m-1}{2} \rfloor}(-1)^i\binom{m-i-1}{i}\binom{m-2i-1;3}{k-2i}\\
 +\sum_{i=0}^{\lfloor \frac{m-2}{2} \rfloor}(-1)^i\binom{m-i-2}{i}\binom{m-2i-2;3}{k-2i}.
     \end{multline*}
and
 \begin{multline*}
   r_{2m+1,k} = \sum_{i=0}^{\lfloor\frac{m}{2}\rfloor} (-1)^i \binom{m-i}{i}\left(\binom{m-2i;3}{k-2i}+ \binom{m-2i;3}{k-2i-1}\right) \\
     -\sum_{i=0}^{\lfloor\frac{m-1}{2} \rfloor} (-1)^i \binom{m-i-1}{i}\left(\binom{m-2i-1;3}{k-2i-1} +\binom{m-2i-1;3}{k-2i-2}\right).
 \end{multline*}

\end{theorem}

   \begin{proof}
     Consider the polynomials $g_n(x)$ defined by
     \[
     \sum_{n\ge 0}g_n(x)z^n =\frac{1}{1-(1+x+x^2)z+x^2z^2}
     \]
     such that
     \[
     g_n(x) =\sum_{i=0}^{\lfloor n/2\rfloor}(-1)^i \binom{n-i}{i}x^{2i}(1+x+x^2)^{n-2i}.
     \]
     In addition, the coefficient of $x^k$ in $g_n(x)$ can be given by
     \[
     [x^k]g_n(x) =\sum_{i=0}^{\lfloor n/2\rfloor}(-1)^i\binom{n-i}{i}x^{2i}(1+x+x^2)^{n-2i} =\sum_{i=0}^{\lfloor n/2\rfloor}(-1)^i\binom{n-i}{i}\binom{n-2i;3}{k-2i}.
     \]
     Then, since $A_m(x)=z +g_m(x)-g_{m-1}(x) +g_{m-2}(x)$, we have
     \begin{align*}
       r_{2m,k} &= [x^k]A_m(x)=[x^k]z +[x^k]g_m(x)-[x^k]g_{m-1}(x) +[x^k]g_{m-2}(x)\\
                &= \delta_{1,m}\delta_{0,k}+ \sum_{i=0}^{\lfloor \frac m2 \rfloor}(-1)^i\binom{m-i}{i} \binom{m-2i;3}{k-2i} -\sum_{i=0}^{\lfloor \frac{m-1}{2} \rfloor}(-1)^i\binom{m-i-1}{i}\binom{m-2i-1;3}{k-2i}\\
                 &~+\sum_{i=0}^{\lfloor \frac{m-2}{2} \rfloor}(-1)^i\binom{m-i-2}{i}\binom{m-2i-2;3}{k-2i}.
     \end{align*}
     Thus, $r_{2m+1,k}$ is obtained from $B_m(x)=g_m(x)+xg_m(x)-xg_{m-1}(x)-x^2g_{m-1}(x)$ in the same way.
   \end{proof}

We can obtain the generating function of $R_n(x)$ from Theorem~\ref{th:gf-AB}.

\begin{theorem}\label{th:rnx}
	The generating function of $R_n(x)$ is
	\[
	\sum_{n\ge 0}R_n(x)y^n = \frac{1+(1+x)y-(x+x^2)y^3-(x +x^2)y^4+x^2y^6}{1-(1+x+x^2)y^2+x^2y^4}.
	\]
\end{theorem}
\begin{proof}
	By the definition of $A_m(x)$ and $B_m(x)$,
	\begin{align*}
	\sum_{n\ge0}R_n(x)y^n &= \sum_{m\ge 0}A_m(x)y^{2m} + \sum_{m\ge 0}B_m(x)y^{2m+1} \\
	&= \sum_{m\ge 0}A_m(x)y^{2m} + y\sum_{m\ge 0}B_m(x)y^{2m} \\
	&= \frac{1-xy^4-x^2y^4+x^2y^6}{1-(1+x+x^2)y^2+x^2y^4} + y\frac{1+x-xy^2-x^2y^2}{1-(1+x+x^2)y^2+x^2y^4} \\
	&= \frac{1+(1+x)y-(x+x^2)y^3-(x +x^2)y^4+x^2y^6} {1-(1+x+x^2)y^2+x^2y^4}.
	\end{align*}
	
\end{proof}

   Since $R_n(1)$ is the number of vertices of $\Phi_{n}$, put $x=1$ in the generating function of $R_n(x)$ we obtain the generating function of the number of vertices of $\Phi_{n}$ is
  \[
  1+y^2 + \frac{2y}{1-y-y^2}.
  \]

  Thus we have the following results related to Fibonacci sequences:

  \begin{corollary}\label{cor:fibs} Using the above notation, we have 
  	$$  \sum_{k=0}^{n} r_{n,k} = 2F_n \;\; (n \geq 3).  $$
  \end{corollary}

\subsection{Cube polynomials}\label{subsec:cube}
  The cube polynomials of $\Phi_n$ is $Q_n(x) = \sum_{k\ge 0}q_{n,k}x^k$, where $q_{n,k} := q_k(\Phi_n)$ is the number of the $k$-dimensional induced hypercubes of $\Phi_{n}$. The first few of $Q_n(x)$ are listed.

  \begin{align*}
  Q_0(x) &= 1 \\
  Q_1(x) &= 2+x \\
  Q_2(x) &= 3+2x \\
  Q_3(x) &= 4+3x \\
  Q_4(x) &= 6+6x+x^2 \\
  Q_5(x) &= 10+13x+4x^2 \\
  Q_6(x) &= 16+25x+11x^2+x^3 \\
  Q_7(x) &= 26+48x+28x^2+5x^3.
\end{align*}

\begin{lemma}[\cite{aWangZY18}]\label{lem:enum}
  Let $L$ be a finite distributive lattice and $K$ a cutting of $L$. Then
  \[
  q_k(L\boxplus K) = q_k(L) + q_k(K) + q_{k-1}(K).
  \]
\end{lemma}

  We can get the recurrence relation of $q_{n,k}$ from Lemma~\ref{lem:enum} evidently.
\begin{proposition}\label{prop:rec-q}
  For $n\ge 4$,
  \[
  q_{n,k} = q_{n-1,k}+q_{n-2,k}+q_{n-2,k-1}.
  \]

\end{proposition}

It is easy to get the recurrence relation of $Q_n(x)$.
\begin{proposition}\label{prop:rec-Q}
	For $n\ge 5$,
	\[
	Q_n(x) = Q_{n-1}(x) + (1+x)Q_{n-2}(x).
	\]
\end{proposition}
\begin{proof}
By the Proposition~\ref{prop:rec-q},
\begin{align*}
  Q_n(x) &= \sum_{k\ge 0}q_{n,k}x^k = \sum_{k\ge 0}(q_{n-1,k} + q_{n-2,k} + q_{n-2,k-1})x^k \\
  &= \sum_{k\ge 0}q_{n-1,k}x^k + \sum_{k\ge 0}q_{n-2,k}x^k + x\sum_{k\ge 0}q_{n-2,k-1}x^{k-1} \\
  &= Q_{n-1}(x) + (1+x)Q_{n-2}(x).
\end{align*}

\end{proof}
\begin{theorem}\label{th:gf-Q}
  The generating function of $Q_n(x)$ is
  \[
  \sum_{n\ge 0}Q_n(x)y^n =\frac{1+(1+x)y-(1+x)^2y^3-(1+x)^2y^4}{1-y-(1+x)y^2}.
  \]
\end{theorem}

%\begin{proof}
%By the Proposition~\ref{prop:rec-Q},
%\begin{align*}
% \sum_{n\ge 0}Q_n(x)y^n &= \sum_{n\ge 5}Q_n(x)y^n+\sum_{n=0}^4Q_n(x)y^n\\
% &= \sum_{n\ge 5}(Q_{n-1}(x)+(1+x)Q_{n-2}(x))y^n+\sum_{n=0}^4Q_n(x)y^n\\
% &= \sum_{n\ge 5}Q_{n-1}(x)y^n+\sum_{n\ge 5}Q_{n-2}(x)y^n+x\sum_{n\ge 5}Q_{n-2}(x)y^n+\sum_{n=0}^4 Q_n(x)y^n\\
% &= y\sum_{n\ge 5}Q_{n-1}(x)y^{n-1}+y^2\sum_{n\ge 5}Q_{n-2}(x)y^{n-2}+xy^2\sum_{n\ge 5}Q_{n-2}(x)y^{n-2}+\sum_{n=0}^4 Q_n(x)y^n\\
% &= y\sum_{n\ge 0}Q_n(x)y^n -y\sum_{n=0}^3 Q_n(x)y^n +y^2\sum_{n\ge 0}Q_n(x)y^n -y^2\sum_{n=0}^1 Q_n(x)y^n
% + xy^2\sum_{n\ge 0}Q_n(x)y^n \\
% &~-xy^2\sum_{n=0}^2 Q_n(x)y^n +\sum_{n=0}^4Q_n(x)y^n\\
% &= (y+y^2+xy^2)\sum_{n\ge 0}Q_n(x)y^n+1+(1+x)y-(1+x)^2y^3-(1+x)^2y^4.
%\end{align*}
%\end{proof}
\begin{proposition}\label{prop:RF-q}
	For $n \ge 0$,
	\[
	Q_n(x)=\sum_{j= 0}^{\lfloor \frac{n+1}{2} \rfloor}\binom{n-j+1}{j}(1+x)^j -
 \sum_{j= 2}^{\lfloor \frac{n+1}{2} \rfloor}\binom{n-j-1}{j-2}(1+x)^j - \sum_{j= 2}^{\lfloor \frac{n}{2} \rfloor}\binom{n-j-2}{j-2}(1+x)^j.
	\]
	and thus
	\[
	q_{n,k}=\sum_{j= 0}^{\lfloor \frac{n+1}{2} \rfloor}\binom{n-j+1}{j}\binom jk -
 \sum_{j= 2}^{\lfloor \frac{n+1}{2} \rfloor}\binom{n-j-1}{j-2}\binom jk - \sum_{j= 2}^{\lfloor \frac{n}{2} \rfloor}\binom{n-j-2}{j-2}\binom jk.
	\]
\end{proposition}

\begin{corollary}
    For $n\ge 3$, the number of vertices of $\Phi_n$ is
     \[
     q_{n,0}= 2\sum_{k=0}^{\lfloor \frac{n-1}{2} \rfloor} \binom{n-1-k}{k} = 2F_n.
     \]
\end{corollary}

\subsection{Maximal cube polynomials}\label{subsec:maxc}

 The maximal cube polynomial of $\Phi_n$ is $H_n(x) = \sum_{k \ge 0} h_{n,k} x^k$, where $h_{n,k} := h_k(\Phi_n)$ be the number of the maximal $k$-dimensional cubes in $\Phi_n$, The first few of $H_n(x)$ are listed.

 \begin{align*}
    H_0(x) &= 1 \\
	H_1(x) &= x \\
    H_2(x) &= 2x \\
	H_3(x) &= 3x \\
	H_4(x) &= 2x+x^2 \\
	H_5(x) &= 4x^2 \\
	H_6(x) &= 5x^2+x^3 \\
    H_7(x) &= 2x^2+5x^3\\
\end{align*}

 We can get the recurrence relation of $h_{n,k}$ from Lemma~\ref{lem:struc}.

\begin{proposition}\label{prop:rec-h}
	For $n \ge 6$,
	\[
	h_{n,k} = h_{n-2,k-1} + h_{n-3,k-1}.
	\]
\end{proposition}

By Proposition~\ref{prop:rec-h} the recurrence relation of $H_n(x)$ is given easily.
\begin{proposition}\label{prop:rec-H}
	For $n \ge 6$,
	\[
	H_n(x) = xH_{n-2}(x) + xH_{n-3}(x).
	\]
\end{proposition}

 The $(1,3,3)$-Padovan number $p_n$ is defined as: $p_0=1$, $p_1=3$, $p_2=3$,$p_n=p_{n-2}+p_{n-3}$, for $n\ge 3$. Hence we have the following corollary.
 \begin{corollary}
  For $n \ge 3$,
  \[
     H_n(1)=p_{n-2}.
  \]
 \end{corollary}

Furthermore, we can obtain the generating function of $H_n(x)$ by Proposition~\ref{prop:rec-H}.
\begin{theorem}\label{th:gf-H}
	The generating function of $H_n(x)$ is
	\begin{align*}
	\sum_{n=0}^\infty H_n(x) y^n &= \frac{1 +xy(1+y) +2xy^3(1+y) -x^2y^3(1+y)^2)}{1-xy^2(1+y)}\\
     &= -2y+xy+xy^2+\frac{1+2y}{1-xy^2(1+y)}.
    \end{align*}
\end{theorem}
%\begin{proof}
%	
% \begin{align*}
%      \sum_{n=0}^\infty H_n(x) y^n &= \sum_{n=6}^\infty H_n(x)+\sum_{n=0}^5 H_n(x)y^n\\
%      &= \sum_{n=6}^\infty(xH_{n-2}+xH_{n-3})y^n+\sum_{n=0}^5 H_n(x)y^n\\
%      &= x\sum_{n=6}^\infty H_{n-2}(x)y^n+x\sum_{n=6}^\infty H_{n-3}(x)y^n +\sum_{n=0}^5 H_n(x)y^n\\
%      &= xy^2\sum_{n=0}^\infty H_n(x)y^n -xy^2\sum_{n=0}^{3}H_n(x)y^n +xy^3\sum_{n=0}^\infty H_n(x)y^n -xy^3\sum_{n=0}^{2}H_n(x)y^n +\sum_{n=0}^{5}H_n(x)y^n \\
%      &= (xy^2+xy^3)\sum_{n=0}^\infty H_n(x)y^n +1 +xy(1+y) +2xy^3(1+y) -x^2y^3(1+y)^2.
% \end{align*}
%\end{proof}

Because $-2y+xy+xy^2$ are parts of $H_1(x)y$ and $H_2(x)y^2$, we have the Proposition~\ref{prop:gf-H}.

\begin{proposition}\label{prop:gf-H}

	For $n\ge 3$,
\[
	H_n(x) =\sum_{k=0}^{\lfloor \frac n2 \rfloor} \left( \binom{k+1}{n-2k} +\binom{k}{n-2k-1} \right)x^k,	
\]
and
\[
h_{n,k} = \binom{k+1}{n-2k} +\binom{k}{n-2k-1}.
\]
\end{proposition}

\begin{proof}

\begin{align*}
  \frac{1+2y}{1-xy^2(1+y)}
  &= \sum_{j\ge 0}(1+y)^jx^jy^{2j} +2\sum_{j\ge 0}(1+y)^{j}x^{j}y^{2j+1}\\
  &= \sum_{j\ge 0}\sum_{n-2j=0}^{j} \binom{j}{n-2j}x^jy^n +2\sum_{j\ge 0}\sum_{n-2j-1=0}^{j} \binom{j}{n-2j-1}x^{j1}y^n\\
  &= \sum_{n\ge 0}\sum_{k=\lceil \frac n3\rceil}^{\lfloor \frac n2 \rfloor}\binom{k}{n-2k}x^ky^n +2\sum_{n\ge 0}\sum_{k =\lceil \frac{n-1}{3}\rceil}^{\lfloor \frac{n-1}{2} \rfloor}\binom{k}{n-2k-1}x^ky^n.
 \end{align*}
 The proof is completed.
\end{proof}

\subsection{Degree sequences polynomials}\label{subsec:deg}

The degree sequences polynomial of $\Phi_n$ is $D_n(x) = \sum_{k\ge 0} d_{n,k} x^k$, where $d_{n,k} := d_k(\Phi_n)$ denoted the number of vertices of the degree $k$ in $\Phi_n$, i.e.\ $d_{n,k} = |\{\,v\in V(\Phi_n) \mid \operatorname{deg}_{\Phi_n}(v)=k\,\}|$. The first few of $D_n(x)$ are listed

\begin{align*}
  D_0(x) &= 1 \\
  D_1(x) &= 2x \\
  D_2(x) &= 2x+x^2 \\
  D_3(x) &= 2x+2x^2 \\
  D_4(x) &= x+4x^2+x^3 \\
  D_5(x) &= 5x^2+4x^3+x^4 \\
  D_6(x) &= 3x^2+9x^3+3x^4+x^5\\
  D_7(x) &= x^2+11x^3+10x^4+3x^5+x^6\\
  %D_8(x) &= 9x^3+19x^4+10x^5+3x^6+x^7.
\end{align*}

  The recurrence relation $d_{n,k}$ is illustrated in the Figure~\ref{fig:mlc-deg}.

\begin{figure}[!htbp]
	\centering
	\begin{tikzpicture}[scale=0.8]
		\draw[rotate=30] (0,0) ellipse (2cm and 1cm)
		(1.8,-2) ellipse (2cm and 1cm)
		(0.7,0.16) ellipse (1.224cm and 0.612cm)
		(1.62,1.58) ellipse (1.224cm and 0.612cm);
		\draw (-1.732+0.33,-1-0.26) -- (0.9+0.28,-1.92-0.18)
		(-1.732-0.3+3.34,-1+0.14+2.15+1.7) -- (-1.732-0.3+3.34,-1+0.14+2.15) -- (0.9-0.28+3.34,-1.92+0.2+2.15)
		(-1.732+0.15+3.34-2,-1-0.09+2.15+1.72-1.37) -- (-1.732+0.15+3.34-2,-1-0.09+2.15+1.72-1.38-1.71);
		\draw[dashed,rotate=30] (3.43,-0.5) ellipse (1.224cm and 0.612cm);
		\draw[dashed] (-1.732-0.3+3.34,-1+0.14+2.15+1.7) -- (0.9-0.15+3.34,-1.92+0.12+2.15+1.72)
		(-1.732+0.15+3.34-2,-1-0.09+2.15+1.72-1.38) -- (0.9+3.34-2,-1.92+2.15+1.72-1.38);
		
		\node at (0.6,2.1) {$\Phi_{n-3}$};
		\node at (2.5,-1) {$\Phi_{n-2}$};
		\node at (1.2,-2.5) {$\Phi_{n}$};
	\end{tikzpicture}
	\caption{Illustrating the recurrence relation of $d_{n,k}$}
	\label{fig:mlc-deg}
\end{figure}
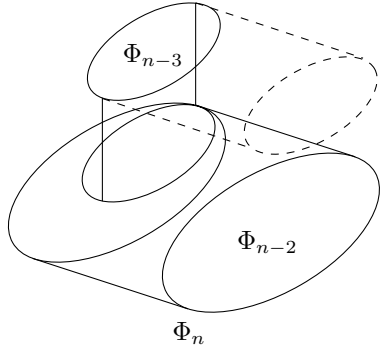

\begin{proposition}\label{prop:rec-d}
For $n\ge 4$,
\[
d_{n,k} = d_{n-2,k-1} + d_{n-1,k-1} - d_{n-3,k-2} + d_{n-3,k-1}.
\]
\end{proposition}

\begin{proposition}\label{prop:rec-D}
For $n\ge 6$,
\[
D_n(x) = xD_{n-1}(x) + xD_{n-2}(x) + (x-x^2)D_{n-3}(x).
\]
\end{proposition}
Therefore, the generating function of $D_n(x)$ is obtained.
\begin{theorem}\label{th:gf-D}
  The generating function of $D_n(x)$ is given by
  \[
  \sum_{n\ge 0}D_n(x)y^n = -y+xy+x^2y^2+\frac{1+y -x^2y^2} {(1-xy)(1-xy^2)-xy^3}.
  \]
\end{theorem}
%\begin{proof}
%
%\begin{align*}
%  \sum_{n\ge 0}D_n(x)y^n &=\sum_{n\ge 6}D_n(x)y^n+\sum_{n=0}^5 D_n(x)y^n\\
%  &=\sum_{n\ge 6}(xD_{n-1}(x)+xD_{n-2}(x)+xD_{n-3}(x)-x^2D_{n-3}(x))y^n
%  +\sum_{n=0}^5 D_n(x)y^n\\
%  &=xy\sum_{n\ge 6}D_{n-1}(x)y^{n-1}+xy^2\sum_{n\ge 6}D_{n-2}(x)y^{n-2}+xy^3\sum_{n\ge 6}D_{n-3}(x)y^{n-3} \\
%  &~-x^2y^3\sum_{n\ge 6}D_{n-3}(x)y^{n-3}+\sum_{n=0}^5 D_n(x)y^n\\
%  &=xy\sum_{n\ge 0}D_n(x)y^n-xy\sum_{n=0}^4D_n(x)y^n +xy^2\sum_{n\ge 0}D_n(x)y^n-xy^2\sum_{n=0}^3D_n(x)y^n\\ &~+xy^3\sum_{n\ge 0}D_n(x)y^n-xy^3\sum_{n=0}^2D_n(x)y^n
%  -x^2y^3\sum_{n\ge 0}D_n(x)y^n +x^2y^3\sum_{n=0}^2D_n(x)y^n +\sum_{n=0}^5D_n(x)y^n\\
%  &= (xy+xy^2+xy^3-x^2y^3)\sum_{n\ge 0}D_n(x)y^n+1+xy+(x-x^2)y^2+(x-x^2-x^3)y^3 +(x-2x^2)y^4\\ &~-(x^3-x^4)y^5.
%\end{align*}
%\end{proof}
\begin{proposition}\label{pro:d}
  For $n\ge 3$, the number of vertices of degree $k$ of $\Phi_n$ is
  \[
  d_{n,k} = \sum_{j=0}^k\left(\binom{n-2j}{k-j}\binom{j}{n-k-j} +\binom{n-2j-1}{k-j}\binom{j}{n-k-j-1} -\binom{n-2j-2}{k-j-2}\binom{j}{n-k-j}\right).
  \]
\end{proposition}

\begin{proof}
The way is similar to the method of \cite{aKlavzMP11}. Using the expansion

\[
\frac{x^n}{(1-x)^{n+1}} =\sum_{j\ge n}\binom{j}{n}x^j,
\]

we consider the formal power series expansion of
 \[
 f(x,y) =\frac{1}{(1-xy)(1-xy^2)-xy^3},
 \]
 \begin{align*}
     f(x,y) &= \frac{1}{(1-xy)(1-xy^2)-xy^3}\\
       &= \frac{(1-xy)^{-1}(1-xy^2)^{-1}}{1-xy^3(1-xy)^{-1}(1-xy^2)^{-1}}\\
       &= \sum_{t \ge 0 }x^t y^{3t}(1-xy)^{-t-1}(1-xy^2)^{-t-1}\\
       &= \sum_{t \ge 0} \frac{(xy)^t}{(1-xy)^{t+1}} \frac{(xy^2)^t}{(1-xy^2)^{t+1}}x^{-t}\\
       &= \sum_{t\ge 0} \sum_{i\ge t} \binom it (xy)^i \sum_{j\ge t}\binom jt (xy^2)^j x^{-t}\\
       &= \sum_{n \ge 0} \sum_{k=0}^n \sum_{j=0}^k \binom{n-2j}{k-j} \binom{j}{n-k-j}x^ky^n.
 \end{align*}

Thus
 \[
 [x^k][y^n]f(x,y) =\sum_{j=0}^k \binom{n-2j}{k-j} \binom{j}{n-k-j}.
 \]
 Note that
 $$
 \frac{1+y-x^2y^2}{(1-xy)(1-xy^2)-xy^3} =f(x,y) +yf(x,y)-x^2y^2f(x,y),
 $$
 then
 \begin{align*}
   &~[x^k][y^n]\frac{1+y-x^2y^2}{(1-xy)(1-xy^2)-xy^3}\\
   &= [x^k][y^n]f(x,y) +[x^k][y^{n-1}]f(x,y) -[x^{k-2}][y^{n-2}]f(x,y)\\
   &= \sum_{j=0}^k \binom{n-2j}{k-j}\binom{j}{n-k-j} +\sum_{j=0}^k \binom{n-2j-1}{k-j}\binom{j}{n-k-j-1}\\ &~-\sum_{j=0}^k \binom{n-2j-2}{k-j-2}\binom{j}{n-k-j}.
 \end{align*}
\end{proof}

  We have a similar result as Corollary \ref{cor:fibs} as follows.
  \begin{corollary} For $n \geq 3$,
  	$$\sum_{k=0}^n d_{n,k}=2F_n.$$
  \end{corollary}

\subsection{Indegree and outdegree polynomials}
  The indegree polynomial of $\Phi_n$ is $D_n^-(x) = \sum_{k=0}^{\lfloor \frac n2 \rfloor} d_{n,k}^-x^k$, where $d_{n,k}^-$ denoted the number of vertices of the indegree $k$ in $\Phi_n$, or the number of anti-chains with only $k$ elements in $\Phi_n$, or the number of the element covered only by $k$ elements in $\Phi_n$. The first few of $D^-_n(x)$ are listed.
  \begin{align*}
	D_0^-(x) &= 1 \\
	D_1^-(x) &= 1+x \\
    D_2^-(x) &= 1+2x \\
	D_3^-(x) &= 1+3x \\
	D_4^-(x) &= 1+4x+x^2 \\
	D_5^-(x) &= 1+5x+4x^2 \\
	D_6^-(x) &= 1+6x+8x^2+x^3 \\
    D_7^-(x) &= 1+7x+13x^2+5x^3.
\end{align*}
\begin{proposition}\label{pro:out-d}
For $n \ge 3$ and $k \ge 1$,
	\[
	d_{n,k}^- = d_{n-1,k}^- + d_{n-2,k-1}^-.
	\]
\end{proposition}
\begin{proposition}\label{pro:ouD}
	For $n \ge 5$,
	\[
	D_n^-(x) = D_{n-1}^-(x) + xD_{n-2}^-(x).
	\]
\end{proposition}

\begin{lemma}[\cite{aWangZY18}]\label{prop:QD-}
	For $n \ge 0$,
	\[
	D_n^-(1+x) = Q_n(x).
	\]
\end{lemma}

By Propositions \ref{pro:out-d},~\ref{pro:ouD} and Lemma~\ref{prop:QD-}, a similar argument as proof of Theorem~\ref{th:gf-AB} gives the generating function of $D^-_n(x)$ as follows.
\begin{theorem}
	The generating function of $D_n^-(x)$ is
	\[
	\sum_{n\ge 0} D_n^-(x)y^n = \frac{1+xy-x^2y^3-x^2y^4}{1-y-xy^2}.
	\]
\end{theorem}

%\begin{proof}	
%   \begin{align*}
%      \sum_{n=0}^\infty D_n^-(x)y^n &= \sum_{n=5}^\infty D_n^-(x)y^n+\sum_{n=0}^4 D_n^-(x)y^n\\
%      &= \sum_{n=5}^\infty(D_{n-1}^-(x)+xD_{n-2}^-(x))y^n+
%      \sum_{n=0}^4 D_n^-(x)y^n\\
%      &= \sum_{n=5}^\infty D_{n-1}^-(x)y^n+x\sum_{n=5}^\infty D_{n-2}^-(x)y^n+ \sum_{n=0}^4 D_n^-(x)y^n\\
%      &= y\sum_{n=0}^\infty D_n^-(x)y^n -y\sum_{n=0}^3 D_n^-(x)y^n+ xy^2\sum_{n=0}^\infty D_n^-(x)y^n-xy^2\sum_{n=0}^2 D_n^-(x)y^n +\sum_{n=0}^4 D_n^-(x)y^n\\
%      &= (y+xy^2)\sum_{n=0}^\infty D_n^-(x)y^n +1+xy-x^2y^3-x^2y^4.
%   \end{align*}
%\end{proof}

\begin{proposition}\label{pro:ou-d}
	For $n \ge 0$,
	\[
	D_n^-(x) =\sum_{k=0}^{\lfloor \frac{n}{2}\rfloor} \left( \binom{n-k-2}{k-1}+\binom{n-k}{k}\right)x^k.
	\]
	Thus
	\[
	d_{n,k}^- =  \binom{n-k-2}{k-1}+\binom{n-k}{k}.
	\]
\end{proposition}

  We have a similar result as Corollary \ref{cor:fibs} on indegree of each vertex of the Fibonacci-like cubes from the fence-like posets.
\begin{corollary} For $n \geq 3$,
	$$\sum_{k=0}^n d_{n,k}^- = \sum_{k=0}^n \left( \binom{n-k-2}{k-1}+\binom{n-k}{k} \right) = 2F_n.$$
\end{corollary}

The conclusion of outdegree is similar to those of indegree.


\begin{thebibliography}{10}

\bibitem{bComte74}
L.~Comtet.
\newblock {\em Advanced Combinatorics: The art of finite and infinite
  expansions}.
\newblock D. Reidel Publishing Company, Dordrecht, enlarged edition, 1974.

\bibitem{aHsu93}
W.~J. Hsu.
\newblock {Fibonacci} cubes-a new interconnection topology.
\newblock {\em IEEE Trans. Parallel Distrib. Syst.}, 4(1):3--12, Jan. 1993.

\bibitem{aKlavzMP11}
S.~Klav\v{z}ar, M.~Mollard, and M.~Petkov\v{s}ek.
\newblock The degree sequence of {Fibonacci} and {Lucas} cubes.
\newblock {\em Discrete Math.}, 311(14):1310--1322, 2011.

\bibitem{aKlavM12}
S.~Klav{\v{z}}ar and M.~Mollard.
\newblock Cube polynomial of {Fibonacci} and {Lucas} cubes.
\newblock {\em Acta Appl Math}, 117(1):93--105, 2012.

\bibitem{aMunarCS2001}
Munarini, E., Cippo, C.~P., and Salvi, N.~Z.
\newblock On the {Lucas} cubes.
\newblock {\em {Fibonacci} Quarterly 39}, 1 (2001), 12--21.

\bibitem{aMolla12}
M.~Mollard.
\newblock Maximal hypercubes in {Fibonacci} and {Lucas} cubes.
\newblock {\em Discrete Appl. Math.}, 160(16):2479--2483, 2012.

\bibitem{aMunarZ02b}
E.~Munarini and N.~Zagaglia~Salvi.
\newblock On the rank polynomial of the lattice of order ideals of fences and
  crowns.
\newblock {\em Discrete Math.}, 259(1):163--177, 2002.

\bibitem{aKlavz2013}
Klav\v{z}ar, S.
\newblock Structure of {Fibonacci} cubes: a survey.
\newblock {\em J Combin Optim 25}, 4 (2013), 505--522.

\bibitem{aSaygE18}
E.~Sayg{\i} and {\"O}.~E\u{g}ecio\u{g}lu.
\newblock $q$-counting hypercubes in {Lucas} cubes.
\newblock {\em Turkish J. Math.}, 42(1):190--203, 2018.

\bibitem{Sloan19}
N.~J.~A. Sloane.
\newblock On-line encyclopedia of integer sequences.
\newblock http://oeis.org/, 2019.

\bibitem{bStanl11}
R.~P. Stanley.
\newblock {\em Enumerative Combinatorics: Volume 1}, volume~49 of {\em
  Cambridge studies in advanced mathematics}.
\newblock Cambridge University Press, Cambridge, 2nd edition, 2011.

\bibitem{aWangZY18}
X.~Wang, X.~Zhao, and H.~Yao.
\newblock Convex expansion for finite distributive lattices with applications.
\newblock {\em arXiv: 1810.06762}, 2018.

\bibitem{aWangZY18c}
X.~Wang, X.~Zhao, and H.~Yao.
\newblock Structure and enumeration results of matchable {Lucas} cubes.
\newblock {\em arXiv: 1810.07329}, 2018.
\end{thebibliography}
\end{document}